\documentclass[12pt,twoside,a4paper]{amsart}
\usepackage{times,amssymb,hyperref,a4wide}

\newtheorem{theorem}{Theorem}[section]
\newtheorem{definition}[theorem]{Definition}
\newtheorem{proposition}[theorem]{Proposition}

\newtheorem{lemma}[theorem]{Lemma}
\theoremstyle{remark}
\newtheorem*{remark}{Remark}

\newcommand{\la}{\langle}
\newcommand{\ra}{\rangle}
\newcommand{\C}{\mathbb{C}}
\newcommand{\dbar}{\overline{\partial}}
\renewcommand{\Re}{\operatorname{Re}}

\begin{document}
\title[Pointwise estimates for the Bergman kernel]{Pointwise estimates for the
Bergman kernel of the weighted Fock space}
\author {Jordi Marzo}
\address{Department of Mathematical Sciences,
Norwegian University of Science and Technology,
N-7491 Trondheim, Norway}
\email{jordi.marzo@math.ntnu.no}

\author{Joaquim Ortega-Cerd\`a}
\address{Departament de Matem\` atica Applicada i An\`alisi,
Universitat de Barcelona, Gran Via 585, 08007-Barcelona, Spain}
\email{jortega@ub.edu}


\thanks{Supported by projects MTM2008-05561-C02-01 and
2005SGR00611}

\date{\today}
\begin{abstract}
We prove upper pointwise estimates for the Bergman kernel of the weighted Fock
space of entire functions in $L^2(e^{-2\phi})$ where $\phi$ is a subharmonic
function with $\Delta \phi$ a doubling measure. We derive estimates for the
canonical solution operator to the inhomogeneous Cauchy-Riemann equation and we
characterize the compactness of this operator in terms of $\Delta \phi$.
\end{abstract}

\maketitle

\section{Introduction}
Let $\phi$ be a subharmonic function in $\C$ whose Laplacian $\Delta \phi$ is a
doubling measure. For $1\le p<\infty$, we consider the Fock spaces 
\[
\mathcal{F}_{\phi}^p=\Bigl\{ f\in \mathcal{H}(\C):\| f
\|_{\mathcal{F}_{\phi}^p}^p=\int_{\C}|f(z)|^p e^{-p\phi(z)}\, dm(z)<\infty
\Bigr\},
\]
and
\[
\mathcal{F}_{\phi}^\infty=\left\{ f\in \mathcal{H}(\C):\| f
\|_{\mathcal{F}_{\phi}^\infty}=\sup_{z\in \C}|f(z)|e^{-\phi(z)}<\infty \right\},
\]
where $dm$ denotes the Lebesgue measure in $\C$.

Let $\overline{K(z,\zeta)}=K_z (\zeta)$ denote the Bergman kernel for
$\mathcal{F}_{\phi}^2$, i.e. for any $f\in \mathcal{F}_{\phi}^2$ 
\[
f(z )=\la f, K_z \ra_{\mathcal{F}_{\phi}^2}= \int_\C f(\zeta) K(z,\zeta)
e^{-2\phi(\zeta)}\, dm(\zeta),\quad z\in \C.
\]

If $\mu=\Delta\phi$, the function $\rho_\phi (z)$ (or simply $\rho(z)$)
denotes the positive radius such that $\mu (D(z,\rho(z)))=1$. The function
$\rho^{-2}$ can be considered as a regularized version of $\Delta\phi$, see
\cite{Chr91} or \cite{MMO03}. We write $D^r (z)=D(z,r\rho (z))$ and 
$D^1(z)=D(z)$ (we will write $D^r_\phi (z)$ if we need to stress the dependence
on $\phi$).

In this context the Bergman kernel has already been studied. In \cite{Chr91}
M. Christ obtained pointwise estimates under the hipothesis that $\phi$ is a
subharmonic function such that
$\mu=\Delta\phi$ is a doubling measure and 
\begin{equation}\label{estricta}
\inf_{z\in \C}\mu(B(z,1))>0.
\end{equation}
This result was extended to several complex variables by H.~Delin and
N.~Lindholm in \cite{Del98} and \cite{Lin01} under similar hypothesis. They
obtain a very fast decay of the Bergman kernel away from the diagonal. 

We will remove hypothesis \eqref{estricta} (which in somes sense  is related to
the strict pseudoconvexity) and keep only the doubling condition (that is
morally closer to finite-type). We still obtain some decay away from the
diagonal, we derive estimates for the canonical solution operator to the
inhomogeneous Cauchy-Riemann equation and we characterize the compactness of
this operator in terms of $\Delta \phi$. Our main result is the following
estimate.

\begin{theorem}\label{proposition_main_estimate}
Let $K(z,\zeta)$ be the Bergman kernel for $\mathcal{F}_{\phi}^2$. There exist
positive constants $C$ and $\epsilon$ (depending only on the doubling constant
for $\Delta \phi$) such that for any $z,\zeta \in \C$
\begin{equation}\label{main_estimate}
|K(z,\zeta)|\le C \frac{1}{\rho(z)\rho(\zeta)}
\frac{e^{\phi(z)+\phi(\zeta)}}{\exp \bigl(
\frac{|z-\zeta|}{\rho(z)}\bigr)^{\epsilon}}.
\end{equation}
\end{theorem}

Although the estimate above seems to be asymmetric in the variables $z,\zeta$
one can see that for $|z-\zeta|<C\max\{ \rho(z),\rho(\zeta) \}$ the values of
$\rho(z)$ and $\rho(\zeta)$ are comparable, see
Lemma~\ref{equivalent_point_rho}. Also when $|z-\zeta|\ge C\max\{
\rho(z),\rho(\zeta) \}$ one can use Lemma~\ref{lemma_distancias} to see that the
same estimate holds with $\rho(\zeta)$ inside the exponential for a different
positive exponent $\epsilon$ (this new exponent depending only on the doubling
constant for $\Delta\phi$). The symmetry becomes apparent when we write
\eqref{main_estimate} in terms of the distance $d_\phi$ induced by the metric
$\rho_\phi^{-2}(z) dz \otimes d\bar{z}$. Indeed, by using
Lemma~\ref{lemma_distancias} one can write \eqref{main_estimate} as
\begin{equation}
|K(z,\zeta)|\le C \frac{1}{\rho(z)\rho(\zeta)}
\frac{e^{\phi(z)+\phi(\zeta)}}{\exp \left(d_\phi (z,\zeta)^{\epsilon} \right)},
\end{equation}
for some $\epsilon>0$ (different from the previous one but still positive). 
The estimate proved in \cite{Chr91} for the Bergman kernel of
$\mathcal{F}_{\phi}^2$ defined for a $\phi$ with doubling Laplacian and
satisfying \eqref{estricta} is
\begin{equation*}
|K(z,\zeta)|\le C \frac{1}{\rho^2(z)} \frac{e^{\phi(z)+\phi(\zeta)}}{\exp
\left( \epsilon d_\phi (z,\zeta) \right)},
\end{equation*}
for some $\epsilon>0$ and all $z,\zeta \in \C$.

Let $N$ be the canonical solution operator to $\dbar$, i.e. $\dbar Nf=f$ and
$Nf$ is of minimal $L^2 (e^{-2\phi})$ norm and let $C(z,\zeta)$ be the integral
kernel such that 
\[
Nf(z)=\int_{\C}e^{\phi(z)-\phi(\zeta)}C(z,\zeta)f(\zeta)\, dm(\zeta).
\]
The boundedness and compactness of this canonical solution operator from $L^2
(e^{-2\phi})$ to itself has been extensively studied in one and several
variables; for a survey on this problem and its applications see \cite{FS01}. It
is shown in \cite{Has06} that for weights on the class considered by M.~Christ,
the condition $\rho(z)\to 0$ when $|z|\to \infty$ is sufficient for compactness.
In the same paper it is shown that the canonical solution operator with
$\phi(z)=|z|^2$ fails to be compact, all these results are contained in
Theorem~\ref{theorem_compactness}. Finally, in \cite{HH07} the authors prove a
result similar to Theorem~\ref{theorem_compactness} with some extra
regularity conditions on $\Delta \phi$.

With Theorem~\ref{proposition_main_estimate} we obtain a pointwise estimate 
on the kernel of the canonical solution operator.

\begin{theorem}\label{theo_estimate_kernels}
There exists an integral kernel $G(z,\zeta)$ such that 
\[
u(z)=\int_{\C}e^{\phi(z)-\phi(\zeta)}G(z,\zeta)f(\zeta)\, dm(\zeta), 
\]
solves $\dbar u=f$ and
\[
|G(z,\zeta)|\lesssim \left\{\begin{array}{ccc} |z-\zeta|^{-1}, &  &
|z-\zeta|\le \rho(z), \\
\rho^{-1}(z)\exp(-d_\phi(z,\zeta)^\epsilon) , & & |z-\zeta|\ge \rho(z).
\end{array} \right.
\]
Moreover, the integral kernel $C(z,\zeta)$ giving the canonical solution to
$\dbar$ in $L^2(e^{-2\phi})$ has the same estimate (with a different exponent
$\epsilon>0$).
\end{theorem}

One can compare this result with the estimate on \cite[Theorem 1.13]{Chr91}
where the author proves that
\begin{equation}\label{estChrist}
|C(z,\zeta)|\lesssim \left\{\begin{array}{ccc} |z-\zeta|^{-1}, &  & |z-\zeta|\le
\rho(z), \\ \rho^{-1}(z)\exp(-\epsilon d_\phi(z,\zeta)) , & & |z-\zeta|\ge
\rho(z). \end{array} \right. 
\end{equation}

As an application of the estimate \eqref{main_estimate} we characterize the
compactness of the canonical solution operator to $\dbar$ in terms of the
measure $\Delta \phi$.

\begin{theorem}\label{theorem_compactness}
Let $\phi$ be a subharmonic function such that $\Delta \phi$ is doubling. The
canonical solution operator $N$ of minimal norm in $L^2(e^{-2\phi})$ to the
inhomogeneous $\dbar$-equation defines a bounded compact operator from
$L^2(e^{-2\phi})$ to itself if and only if $\rho_\phi(z)\to 0$ when $|z|\to
\infty$.
\end{theorem}

Any of the estimates on $C(z,\zeta)$ (the estimate in
Theorem~\ref{theo_estimate_kernels} or the result by Christ, \eqref{estChrist})
can be used in order to prove this theorem, because as soon as one supposes the
compactness of the canonical solution operator $N$, the function $\rho$ turns
out to be bounded and therefore \eqref{estricta} holds.

There is some natural gain (or loss) in the H\"ormander estimates if the
Laplacian of $\phi$ is big (or small). If we  incorporate the Laplacian in the
weight then we always get boundedness, under some  mild regularity assumption
(the doubling property) but we never get compactness:

\begin{proposition} \label{proposition_bounded_no_compact}
Let $\phi$ be a subharmonic function such that $\Delta \phi$ is doubling. The
solution $u$ to the equation $\dbar u=f$ of minimal norm in $L^2(e^{-2\phi})$ is
such that $\| u e^{-\phi} \|_{L^p(\C)}\lesssim \| f e^{-\phi} \rho
\|_{L^p(\C)}$, for all $p\in[1,\infty]$. Moreover, the solution operator $ N$
acting from $L^2(e^{-2\phi}\rho^2)$ to $L^2(e^{-2\phi})$ is always bounded but
it is never compact.
\end{proposition}

\begin{remark}The first statement in this proposition has been proved already in
\cite[Theorem~C]{MMO03} by using peak functions instead of estimates for the
Bergman kernel.
\end{remark}

\section{Preliminaires}

In this section we collect some material from \cite{Chr91} and \cite{MMO03} that
will be used along the proofs and we deduce some easy estimates for the Bergman
kernel near the diagonal.

\begin{definition}
A nonnegative Borel measure $\mu$ is called doubling if there exists $C>0$ such
that
\[
\mu(D(z,2r))\le C \mu (D(z,r))
\]
for all $z\in \C$ and $r>0$. The smallest constant $C$ in the previous
inequality is called the doubling constant for $\mu$.
\end{definition}

\begin{lemma}{\cite[Lemma 2.1]{Chr91}}\label{comparacio}
Let $\mu$ be a doubling measure in $\C$. There exists a constant $\gamma>0$ such
that for any disks $D,D'$ with respective radius $r>r'$ and with $D\cap D'\neq
\emptyset$ 
\[
\left(\frac{\mu(D)}{\mu(D')}\right)^\gamma\lesssim
\frac{r}{r'}\lesssim \left(\frac{\mu(D)}{\mu(D')}\right)^{1/\gamma}.
\]
\end{lemma}

\begin{remark} 
In particular for any $z\in \C$ and $r>1$ there exists a constant $\gamma>0$
(depending only on the doubling constant for $\mu$) such that
\begin{equation}\label{cota_medida_bola}
    r^\gamma\lesssim \mu(D^r(z)) \lesssim r^{1/\gamma}.
\end{equation}
\end{remark}

It follows inmediately from Lemma~\ref{comparacio} that the function $\rho$ is
nearly constant on balls.

\begin{lemma}\label{equivalent_point_rho}
If $D(z)\cap D(\zeta)\neq \emptyset$ then $\rho(z)\sim \rho(\zeta)$, with
constants depending only on the doubling constant for $\Delta \phi$.
\end{lemma}

\begin{remark}
There exist constants $\eta, C>0$ and $0<\beta <1$ such that
\[
\frac{C^{-1}}{|z|^\eta}\le \rho(z)\le C |z|^\beta
\] 
for $|z|>1$, \cite[Remark~1]{MMO03}. 
\end{remark}

The following lemma shows that our main estimate \eqref{main_estimate} is
symmetric in the variables $z,\zeta$.

\begin{lemma}{\cite[p. 205]{Chr91}}\label{bounded_quotient}
If $\zeta \not\in D(z)$ then
\[
\frac{\rho(z)}{\rho(\zeta)}\lesssim \left(
\frac{|z-\zeta|}{\rho(\zeta)}\right)^{1-\delta}
\]
for some $0<\delta<1$ depending only on the doubling constant for $\Delta\phi$.
\end{lemma}

\begin{definition}
Given $z,\zeta \in \C$
\[
d_\phi(z,\zeta)=\inf_\gamma \int_0^1 |\gamma' (t)| \frac{dt}{\rho(\gamma(t))},
\]
where $\gamma$ runs on the piecewise $\mathcal{C}^1$ curves
$\gamma:[0,1]\to \C$ with $\gamma(0)=z$ and $\gamma(1)=\zeta$.
\end{definition}

The following lemma was proved in \cite[Lemma 4]{MMO03}.

\begin{lemma}\label{lemma_distancias}
There exists $\delta>0$ such that for every $r>0$ there exists $C_r>0$ such that
\[
C_r^{-1}\frac{|z-\zeta|}{\rho(z)}\le d_\phi(z,\zeta)\le C_r
\frac{|z-\zeta|}{\rho(z)},\quad\text{for }\zeta\in D^r(z),
\]
and
\[
C_r^{-1}\left( \frac{|z-\zeta|}{\rho(z)}\right)^\delta\le
d_\phi(z,\zeta)\le C_r
\left(\frac{|z-\zeta|}{\rho(z)}\right)^{2-\delta},\quad\mbox{for }\zeta\in
D^r(z)^c.
\]
\end{lemma}

The following lemma will be used repeatedly in what follows.
\begin{lemma}\label{lemma_integrability}
Let $\phi$ be a subharmonic function with $\mu=\Delta \phi$ doubling. Then for
any $\epsilon>0$ and $k\ge 0$ 
\[
\int_\C \frac{|z-\zeta|^k}{\exp
d_\phi(z,\zeta)^\epsilon} d\mu(z)\le C\rho^k(\zeta),
\] 
where $C>0$ is a constant depending only on $k$, $\epsilon$, and on the doubling
constant for $\mu$.
\end{lemma}

\begin{proof}
Let $f(t)=\frac{k}{\epsilon}t^{\frac{k}{\epsilon}-1}-t^{\frac{k}{\epsilon}}$
then for any $x>0$
\[
\int_x^{+\infty} e^{-t}f(t)=e^{-x}x^{k/\epsilon},
\]
and
\begin{align*}
\int_\C & \frac{|z-\zeta|^k}{\exp d_\phi(z,\zeta)^\epsilon} d\mu(z)\lesssim
\rho^k(\zeta )\mu(D(\zeta))+
\int_{D(\zeta)^c} \rho^k(\zeta ) \int_{\left( \frac{|z-\zeta|}{\rho(\zeta
)}\right)^\epsilon} e^{-t}f(t)dt d\mu(z)\\ 
&\lesssim \rho^k(\zeta )+\rho^k(\zeta)\int_{1}^{+\infty}
e^{-t}f(t)\mu(D^{t^{1/\epsilon}}(\zeta))dt
\lesssim \rho^k(\zeta )\left(1+\int_{1}^{+\infty}
e^{-t}f(t)t^{1/\gamma\epsilon}dt\right).
\end{align*}
\end{proof}

We will also use some Cauchy-type estimates for functions in the space,
\begin{lemma}{\cite[Lemma 19]{MMO03}}\label{lemma_sub_estimate}
For any $r>0$ there exists $C=C(r)>0$ such that for any $f\in \mathcal{H}(\C)$
and $z\in \C:$
\begin{itemize}
\item[(a)] $|f(z)|^2 e^{-2\phi (z)}\le C \int_{D^r(z)}|f(\zeta)|^2 e^{-2\phi
(\zeta)} \frac{dm(\zeta)}{\rho^2 (\zeta)}$.
\item[(b)] $| \nabla (|f| e^{-\phi})(z)|^2 \le C \int_{D^r(z)}|f(\zeta)|^2
e^{-2\phi (\zeta)}\frac{dm(\zeta)}{\rho^2 (\zeta)}$.
\item[(c)] If $s>r$, $|f(z)|^2 e^{-2\phi(z)} \le C_{r,s} \int_{D^s(z)\setminus
D^r(z)}|f(\zeta)|^2 e^{-2\phi (\zeta)} \frac{dm(\zeta)}{\rho^2 (\zeta)}$.
\end{itemize}
\end{lemma}

The following result proved in \cite[Theorem 14]{MMO03} shows that the same
space $\mathcal{F}_\phi^2$ can be defined with a more regular weight.

\begin{proposition}\label{proposition_regular_phi}
Let $\phi$ be a subharmonic function such that $\Delta\phi$ is doubling. There
exists $\widetilde{\phi}\in \mathcal{C}^{\infty}(\C)$ such that
$|\phi-\widetilde{\phi}|\le C$ with $\Delta \widetilde{\phi}$ doubling and
\begin{equation*}
\Delta \widetilde{\phi}\sim \frac{1}{\rho_{\widetilde{\phi}}^{2}}
\sim \frac{1}{\rho_{\phi}^{2}}.
\end{equation*}
\end{proposition}

As a first step in proving Theorem~\ref{proposition_main_estimate}, in the
remainder of the section we derive some estimates for the Bergman kernel on the
diagonal or near the diagonal.

\begin{proposition}\label{prop_diagonal_kernel_estimate}
There exist $C>0$ such that
\begin{equation}\label{diagonal_kernel_estimate}
C^{-1}\frac{e^{2\phi(z)}}{\rho^2(z)}\le K(z,z)\le
C\frac{e^{2\phi(z)}}{\rho^2(z)}
\end{equation}
\end{proposition}

\begin{proof}
Let $z\in \C$ be fixed. For any $M\in \mathbb{N}$ there exists a holomorphic
function $P_z$ such that $P_z(z)=1$ and
\[
|P_z (\zeta)|\lesssim e^{\phi(\zeta)-\phi(z)}\min \left\{ 1 , \left(
\frac{\rho(z)}{|z-\zeta|}\right)^M \right\},
\]
see \cite[Appendix]{MMO03}. For some $c_0>0$ (to be determined) we define the
entire function
\[
f_{z}(\zeta)=c_0 \frac{e^{\phi(z)}}{\rho(z)}P_z(\zeta).
\]
Then
\begin{align*}
\int_{\C} & |f_z (\zeta)|^2 e^{-2\phi(\zeta)}dm(\zeta) \le C c_0^2+\int_{D(z)^c
}\left( \frac{\rho(z)}{|z-\zeta|}\right)^{2M} \frac{dm(\zeta)}{\rho^2(z)} =C
c_0^2(1+\frac{\pi}{M-1})\le 1
\end{align*}
for $c_0$ small enough. For such a fixed $c_0$ we have $f_{z}(z)=c_0 e^{\phi(z)}
\rho^{-1}(z)$ and therefore
\[
K(z,z)=\sup \{ |f(z)|^2: f\in \mathcal{F}_{\phi}^2, \| f
\|_{\mathcal{F}_{\phi}^2}\le 1 \}\gtrsim \frac{e^{2\phi(z)}}{\rho^2 (z)}.
\]
The other estimate follows by using the reproducing property for the Bergman
kernel, Lemma~\ref{equivalent_point_rho} and inequality (a) in
Lemma~\ref{lemma_sub_estimate}, see the next proposition, where this is done in
detail.
\end{proof}

The following coarse estimate will give us \eqref{main_estimate} when the points
$z,\zeta\in \C$ are close to each other.

\begin{proposition}\label{prop_coarse_estimate}
Let $K(z,\zeta)$ be the Bergman kernel for $\mathcal{F}_{\phi}^2$. Then there
exists $C>0$ (depending only on the doubling constant for $\Delta \phi$) such
that for any $z,\zeta \in \C$
\begin{equation}\label{coarse_estimate}
|K(z,\zeta)|\le C \, \frac{e^{\phi(z)+\phi(\zeta)}}{\rho(z)\rho(\zeta)}.
\end{equation}
\end{proposition}

\begin{proof}
Let $z\in \C$ be fixed. Applying (a) in Lemma~\ref{lemma_sub_estimate} to
the reproducing kernel $K_z$ and using Lemma~\ref{equivalent_point_rho}
\[
\begin{split}
|K_z(\zeta)|^{2}e^{-2\phi(\zeta)}\lesssim &
\int_{D(\zeta)}|K_z(w)|^{2}e^{-2\phi(w)}\frac{dm(w)}{\rho^2(w)}\lesssim\\
\lesssim &\int_{\C}|K_z(w)|^{2}e^{-2\phi(w)}\frac{dm(w)}{\rho^2(\zeta)}= 
\frac{K(z,z)}{\rho^2(\zeta)}.
\end{split}
\]
Finally, by using Proposition~\ref{prop_diagonal_kernel_estimate} the estimate
follows.
\end{proof}

\section{Proof of Theorem~\ref{proposition_main_estimate}}

We will follow  a similar argument as in \cite{Lin01} when Lindholm 
studies the case when $\Delta \phi$ is bounded. In fact the basic trick goes
back to Kerzman in \cite{Ker72}, where the Bergman kernel is estimated using
the estimates on the solution to an inhomogeneous Cauchy-Riemann equation.

We are interested in studying the behaviour of $K(z,\zeta)$ when the points
$z,\zeta$ are far apart.

Let $z,\zeta\in \C$ be fixed points such that $D(z)\cap D(\zeta)=\emptyset$. Let
$0\le \chi\le 1$ be a function in $\mathcal{C}_{c}^{\infty}(\C)$ with
$\mbox{supp} \chi \subset D(\zeta)$ such that $\chi\equiv 1$ in $D^{1/2}(\zeta)$
and
\[
|\overline{\partial}\chi|^{2}\lesssim \frac{\chi}{\rho^{2}(\zeta)}.
\] 
We have that
\begin{align*}
|K_z(\zeta)|^{2}e^{-2\phi(\zeta)} & \lesssim \frac{1}{\rho^{2}(\zeta)}
\int_{D^{1/2}(\zeta)} |K_z(w)|^{2}e^{-2\phi(w)}\, dm(w) \nonumber \\
= & \frac{1}{\rho^{2}(\zeta)} \int_{D^{1/2}(\zeta)} \chi(w)
|K_z(w)|^{2}e^{-2\phi(w)}\, dm(w)\lesssim \frac{1}{\rho^{2}(\zeta)} \| K_z
\|^{2}_{L^{2}(\chi e^{-2\phi})}
\end{align*}
Then, of course $\| K_z \|^{2}_{L^{2}(\chi e^{-2\phi})}=\sup_{f}|\langle
f,K_z\rangle_{L^{2}(\chi e^{-2\phi})}| $ where the supremum runs over all $f$ be
a holomorphic function in $D(\zeta)$ such that 
\[
\int |f|^{2}e^{-2\phi}\chi\,dm=1.
\]

As $f\chi\in L^{2}(e^{-2\phi})$ one has
\[
\la f, K_z \ra_{L^{2}(\chi e^{-2\phi}) }=P(f\chi)(z),
\]
where $P=P_\phi$ stands for the Bergman projection
\[
P_\phi(f)(z)=\int_{\C}K(z,\zeta)f(\zeta)e^{-2\phi (\zeta)}\, dm(\zeta),
\]
which is bounded from $L^2(e^{-2\phi})$ to $\mathcal{F}_\phi^2$. Now
\[
u=f\chi-P(f\chi),
\]
is the canonical solution (in $L^{2}(e^{-2\phi})$) of
\begin{equation}\label{delta}
\dbar u=\dbar (f\chi)=f\dbar\chi,
\end{equation}
and, since $\chi(z)=0$, one has
\[
\left|\la f, K_z \ra_{L^{2}(\chi e^{-2\phi}) } \right|=
\left| P(f\chi)(z) \right|=\left| u(z) \right|.
\]
We use a H\"ormander's type theorem to majorize this last expression by an
integral involving $f\dbar\chi$. One technical difficulty is that our function
$\phi$ is not smooth enough, so first of all we define a regularized version.

Let $0<\epsilon<1$ a constant to be chosen later. Let
\[
\varphi(w)=\left( \frac{|w-\zeta|}{\rho_\phi (\zeta)} \right)^{\epsilon},
\]
(we will write $\varphi_\epsilon$ if we need to stress the dependence on
$\epsilon$). The function $\varphi$ is subharmonic and
\[
\frac{\partial \varphi}{\partial w}(w)=\frac{\epsilon |w-\zeta|^{\epsilon-2}
(\overline{w}-\overline{\zeta})}{2\rho_\phi^{\epsilon} (\zeta)},\qquad\Delta
\varphi(w)=\frac{\epsilon^{2} |w-\zeta|^{\epsilon-2}}{4\rho_\phi^{\epsilon}
(\zeta)}.
\]
Considering the dependence on $\epsilon$ one has
\[
\Delta \varphi_{2\epsilon}(w)=4\left| \frac{\partial
\varphi_{\epsilon}}{\partial w}(w) \right|^{2}.
\]
The Laplacian of $\varphi$ is not bounded above, so we define
\[
\psi=\frac{1}{|B_{\rho_\phi (\zeta)}|}\chi_{\rho_\phi (\zeta)}\ast
\varphi,
\]
where $\chi_{\rho_\phi (\zeta)}=\chi_{B_{\rho_\phi (\zeta)}}$ is the
characteristic function of $B_{\rho_\phi (\zeta)}=B(0,\rho_\phi (\zeta))$.

By H\"older's inequality
\[
\left| \frac{\partial \psi}{\partial w} \right|^{2}\le \frac{1}{|B_{\rho_\phi
(\zeta)}|}\chi_{\rho_\phi (\zeta)}\ast \left| \frac{\partial \varphi}{\partial
w} \right|^{2},
\]
and
\[
\Delta\psi_{2\epsilon}(w)=\left( \frac{1}{|B_{\rho_\phi
(\zeta)}|}\chi_{\rho_\phi (\zeta)}\ast \Delta \varphi_{2\epsilon} \right)(w)=
\left( \frac{1}{|B_{\rho_\phi (\zeta)}|}\chi_{\rho_\phi (\zeta)}\ast 4\left|
\frac{\partial \varphi_{\epsilon}}{\partial w} \right|^{2} \right)(w).
\]
We denote $\Phi=\Delta\psi_{2\epsilon}(w)/4$ and, as before, we will write
$\Phi_{\epsilon}$ if needed.

By \cite[Theorem 14]{MMO03} one can build $\widetilde{\phi}\in
\mathcal{C}^{\infty}$ such that $|\phi-\widetilde{\phi}|\le C$ with $\Delta
\widetilde{\phi}$ doubling and
\begin{equation}\label{dos}
\Delta \widetilde{\phi}\sim \frac{1}{\rho_{\widetilde{\phi}}^{2}} \sim
\frac{1}{\rho_{\phi}^{2}}.
\end{equation}

\begin{lemma}\label{lem1}
There exist $0<\epsilon_{0}<1$ and $0<C_{1},C_{2}<1$ (depending only on the
doubling constant for $\Delta \phi$) such that for any $0<\epsilon\le
\epsilon_{0}$
\[
\left| \frac{\partial \psi_{\epsilon}}{\partial w}(w) \right|^{2}\le C_{1}
\Delta\widetilde{\phi}(w),\quad\text{and} \quad\Delta \psi_{\epsilon}(w)\le
C_{2}\Delta\widetilde{\phi}(w).
\]
\end{lemma}

This Lemma is an easy consequence of the following:

\begin{lemma}\label{lem2}
For any $C>0$ there exists $0<\epsilon_{0}<1$ (depending only on the doubling
constant for $\Delta \phi$ and $C$) such that
\[
\Phi_{\epsilon}(w)\le C \frac{1}{\rho_\phi^{2}(w)}
\]
if $0<\epsilon\le \epsilon_{0}$.
\end{lemma}

\begin{proof}[Lema~\ref{lem2} implies Lemma~\ref{lem1}]
By \eqref{dos} let $C'>0$ such that
\[
\frac{1}{\rho_\phi^{2}(w)}\le C'\Delta\widetilde{\phi} (w).
\] 
Let $\epsilon_{0}>0$ the one provided by Lemma~\ref{lem2} for $C>0$ such that
$4CC'<1$. If $0<\epsilon\le \epsilon_{0}$ we have
\[
\left| \frac{\partial \psi_\epsilon }{\partial w}(w)\right|^{2}\le
\Phi_{\epsilon}(w)\le C\frac{1}{\rho_\phi^{2}(w)}\le  CC^{'}\Delta
\widetilde{\phi} (w),
\]
and
\[
\Delta\psi_\epsilon (w)=4\Phi_{\epsilon/2}(w)\le 4C\frac{1}{\rho_\phi^{2}(w)}\le
 4CC'\Delta \widetilde{\phi} (w).
\]
Then it is enough to take $C_{1}=CC'$ and $C_{2}=4CC'$.
\end{proof}

\begin{proof}[Lemma~\ref{lem2}]
We want to see that for $C>0$ there exists $1>\epsilon_{0}>0$ such that for
$0<\epsilon\le \epsilon_{0}$
\[
\left(\frac{1}{|B_{\rho_\phi (\zeta)}|}\chi_{\rho_\phi (\zeta)}\ast \left|
\frac{\partial \varphi}{\partial w}\right|^{2}\right) (w)\le C
\frac{1}{\rho_\phi^{2}(w)}.
\]

We will split the proof in two cases:

\noindent CASE 1: Suppose that $D(w)\cap D^{2}(\zeta)\neq \emptyset$. The
function $\Phi$ has a maximum in $w=\zeta$ (because $|\partial \varphi /\partial
w|^{2}(u)\sim 1/|u-\zeta|^{2-2\epsilon}$) so it is enough to see that
$\Phi_{\epsilon}(\zeta)\le C \rho_\phi^{-2}(w)$.
\begin{align*}
\Phi_{\epsilon}(\zeta)=& \frac{1}{\pi
\rho^{2}(\zeta)}\int_{D_{\zeta}}\frac{\epsilon^{2}
|\zeta-z-\zeta|^{2\epsilon-2}}{4\rho^{2\epsilon}(\zeta)}\, dm(z)\\
=&\frac{\epsilon^{2}}{4\pi
\rho^{2\epsilon+2}(\zeta)}\int_{D_{\zeta}}|z|^{2\epsilon-2}\, dm(z)=
\frac{\epsilon^{2}}{4\pi \rho^{2\epsilon+2}(\zeta)}
\int_{0}^{\rho(\zeta)}\int_{0}^{2\pi} t^{2\epsilon-1}\, dt d\theta=
\frac{\epsilon}{4 \rho^{2}(\zeta)}.
\end{align*}
so we need
\begin{equation*}
\frac{\epsilon}{4}\lesssim \left( \frac{\rho(\zeta)}{\rho(w)}\right)^{2},
\end{equation*}
and this property holds for $0<\epsilon\le \epsilon_{0}$ because
$\rho(\zeta)\sim \rho(w)$.

\noindent CASE 2: Suppose that $D(w)\cap D^{2}(\zeta)=\emptyset$.
\begin{align*}
\Phi_{\epsilon}(w) & =\frac{1}{|B_{\rho_\phi
(\zeta)}|}\int_{D_{\zeta}}\frac{\epsilon^{2}|w-z-\zeta|^{2\epsilon-2}}{4\rho^{
2\epsilon}(\zeta)}\, dm(z)\\
&=\frac{\epsilon^{2}}{4|B_{\rho_\phi
(\zeta)}|\rho^{2\epsilon}(\zeta)}\int_{B_{\rho_\phi (\zeta)}}
|w-u|^{2\epsilon-2}\, dm(u)\le  
\frac{\epsilon^{2}2^{2-2\epsilon}}{4\rho^{2\epsilon}(\zeta)|w-\zeta|^{
2-2\epsilon}}.
\end{align*}

So we need
\[
\frac{\epsilon^{2}|w-\zeta|^{2\epsilon-2}}{2^{2\epsilon}\rho^{2\epsilon}
(\zeta)}\lesssim \frac{1}{\rho^{2}(w)}.
\]
or equivalently
\begin{equation}\label{to}
\frac{2^{\epsilon}C}{\epsilon}\left( \frac{|w-\zeta|}{\rho(\zeta)}
\right)^{1-\epsilon}\ge \frac{\rho(w)}{\rho(\zeta)},
\end{equation}
and this follows from Lemma~\ref{bounded_quotient} because $\zeta\not\in D(w)$.
We would like to mention that, as $|w-\zeta|>\rho(\zeta)$, the last inequality
holds also for any exponent smaller than $\delta$. Finally, as
$2^{\epsilon}C/\epsilon$ goes to infinity when $\epsilon\to 0$, one can find
$\epsilon_{0}$ such that \eqref{to} holds for any $0<\epsilon\le \epsilon_{0}$.
\end{proof}

From now on we will fix $\varepsilon>0$ in such a way that the conclusions of
Lemma~\ref{lem1} do hold. The following Lemma is an easy consequence of the
previous ones.

\begin{lemma}\label{lem3}
For $\varrho=\widetilde{\phi}-\psi$, one has
\[
\Delta\varrho\sim \Delta\widetilde{\phi}, \quad\mbox{and}\quad
\frac{1}{\rho_{\varrho}^{2}}\sim \frac{1}{\rho_{\widetilde{\phi}}^{2}}.
\]
\end{lemma}

\begin{proof}
As $\psi$ is subharmonic $\Delta\widetilde{\phi}\ge
\Delta\widetilde{\phi}-\Delta\psi=\Delta\varrho$. The other inequality follows
from Lemma~\ref{lem1} since $\Delta\varrho\ge (1-C_{2})\Delta\widetilde{\phi}$,
with $0<C_{2}<1$. The relation between the corresponding regularization follows
automatically.
\end{proof}

As $D_{\phi}(\zeta)\cap D_{\phi}(z)=\emptyset$, the function $f\chi$ vanishes
off $D_{\phi}(\zeta)$ and therefore (recall that by Lemma~\ref{lem3}
$\rho_{\varrho}\sim \rho_{\widetilde{\phi}}\sim\rho_{\phi}$) the function $u$ is
holomorphic in $D_{\varrho}^{r}(z)$ for some $r>0$ again by (a) in
Lemma~\ref{lemma_sub_estimate}
\begin{align}\label{hol}
|u(z)|^{2}e^{-2\phi(z)+2\psi(z)} & \lesssim
|u(z)|^{2}e^{-2\widetilde{\phi}(z)+2\psi(z)}=|u(z)|^{2}e^{-2\varrho(z)}
\nonumber\\
&\lesssim  
\int_{D_{\varrho}^{r}(z)}|u(w)|^{2}e^{-2\varrho(w)}\frac{dm(w)}{\rho^{2}_{
\varrho}(w)}\lesssim
\frac{1}{\rho^{2}_{\varrho}(z)}\int_{\C}|u(w)|^{2}e^{-2\varrho(w)}\, dm(w)
\nonumber\\
&\sim \frac{1}{\rho^{2}_{\phi}(z)}\int_{\C}|u(w)|^{2}e^{-2\varrho(w)}\, dm(w).
\end{align}

We estimate this last integral using the classical H\"ormander theorem:

\begin{theorem}[H\"ormander]\label{hormander_theorem}
Let $\Omega \subset \C$ be a domain and $\phi\in\mathcal{C}^{2}(\Omega)$ be
such that $\Delta \phi \ge 0$. For any $f\in L^{2}_{loc}(\Omega)$ there exist a
solution $u$ to $\dbar u=f$ such that
\[
\int |u|^{2}e^{-2\phi} \le \int \frac{|f|^{2}}{\Delta \phi}e^{-2\phi}. 
\]
\end{theorem}
and also with a variant due to Berndtsson (see \cite[Lemma 2.2]{Ber01}):
\begin{theorem}\label{bo}
If 
\[
\left| \frac{\partial \psi}{\partial w}\right|^{2}\le C_{1}\Delta
\widetilde{\phi}, \quad\mbox{with}\quad 0<C_{1}<1,
\] 
and for any $g$  one can find $v$ such that $\dbar v=g$ with
\begin{equation}\label{ber}
\int |v|^{2} e^{-2\phi-2\psi} \le \int
\frac{|g|^{2}}{\Delta\widetilde{\phi}}e^{-2\phi-2\psi},
\end{equation}
then (for $v_0$) the canonical solution in $L^{2}(e^{-2\phi})$, one has
\[
\int |v_{0}|^{2} e^{-2\phi+2\psi}\le C\int
\frac{|g|^{2}}{\Delta\widetilde{\phi}}e^{-2\phi+2\psi},
\]
where $C=6/(1-C_{1})^{2}$.
\end{theorem}

We know that $\Delta(\widetilde{\phi}+\psi)\ge 0$,then applying
Theorem~\ref{hormander_theorem}, to $\overline{\partial}(f\chi)$,one has $v$
such that $\dbar v=\overline{\partial}(f\chi)$ with
\begin{align*}
\int |v|^{2}e^{-2\widetilde{\phi}-2\psi} & \le \int \frac{|\dbar
v|^{2}}{\Delta(\widetilde{\phi}+\psi)}e^{-2\widetilde{\phi}-2\psi}\\
&\le \int \frac{|\dbar v|^{2}}{\Delta
\widetilde{\phi}}e^{-2\widetilde{\phi}-2\psi}.
\end{align*}
As $|\phi-\widetilde{\phi}|\le C$ we have that \eqref{ber} holds and by
Theorem~\ref{bo}
\[
\int |u|^{2} e^{-2\phi+2\psi}\le C\int \frac{|\dbar
u|^{2}}{\Delta\widetilde{\phi}}e^{-2\phi+2\psi}.
\]
The functions $\phi,\widetilde{\phi}$ are pointwise equivalent and
$\Delta\widetilde{\phi}\sim \rho^{-2}_{\phi}$ so one can estimate \eqref{hol} as
\begin{align}\label{not}
\frac{1}{\rho^{2}_{\phi}(z)}\int |u(w)|^{2} &
e^{-2\varrho(w)}\, dm(w)
\lesssim
\frac{1}{\rho^{2}_{\phi}(z)} \int_{D(\zeta)}
\rho^{2}_{\phi}(w)|\overline{\partial}
(f\chi)(w)|^{2}e^{-2\varrho(w)}\, dm(w)
\nonumber
\\
&
=
\frac{1}{\rho^{2}_{\phi}(z)} \int_{D(\zeta)} \rho^{2}_{\phi}(w)|
f(w)|^{2}\frac{\chi(w)}{\rho^{2}_{\phi}(\zeta)}e^{-2\varrho (w)}\, dm(w).
\end{align}

The function $\psi$ is bounded above in $D(\zeta)$ by a constant depending
only on the doubling constant for
$\Delta\phi$,
indeed, for $w\in D(\zeta)$
\[
\frac{1}{\pi\rho^{2}_{\phi}(\zeta)}\int
\chi_{B_{\rho_{\phi}(\zeta)}}(w-u)\varphi (u)\, dm(u) \le
\frac{1}{\pi\rho^{2}_{\phi}(\zeta)}\int_{D^{2}(\zeta)}\varphi (u)\, dm(u)
\lesssim
2^{\epsilon}.
\]

So finally \eqref{not} can be estimated by
\begin{align*}
\int_{D(\zeta)} & \frac{\rho^{2}_{\phi}(w)|   
f(w)|^{2}\chi(w)}{\rho^{2}_{\phi}(z)\rho^{2}_{\phi}(\zeta)}e^{-2\varrho(w)}
\, dm(w) \lesssim   
\int_{D(\zeta)}\frac{|f(w)|^{2}\chi(w)}{\rho^{2}_{\phi}(z)}e^{-2\phi(w)}\, dm(w)
= \frac{1}{\rho^{2}_{\phi}(z)}.
\end{align*}
and we have
\begin{equation}\label{first_form}
|\, K(\zeta,z)|^{2}\lesssim
\frac{1}{\rho_{\phi}^{2}(z)\rho_{\phi}^{2}(\zeta)}
\frac{e^{2\phi(z)+2\phi(\zeta)}}{e^{2\psi(z)}}
\end{equation}

\subsection{Pointwise estimates}

In this subsection we deduce a new expression, without $\psi$, for
\eqref{first_form}. The new expression is the one appearing in
Theorem~\ref{proposition_main_estimate} and therefore this will finish the
proof.

\begin{lemma}
If $D(\zeta)\cap D(w)=\emptyset$ there exists $C>0$ such that
\[
|\psi(w)-\varphi(w)|\le C.
\]
\end{lemma}

\begin{proof}
Using the subharmonicity
\[
\psi(w)-\varphi(w)=\frac{1}{\rho^{\epsilon}(\zeta)}\left\{
\frac{1}{|D(\zeta)|}\int_{D(\zeta)}|w-u|^{\epsilon}\, dm(u)-|w-\zeta|^{\epsilon}
\right\}\ge 0. 
\]

On the other hand, if $|w-\zeta|\le 2\rho(\zeta)$ it is plain that
\[\psi(w)=\frac{1}{\rho^{\epsilon}(\zeta)}
\frac{1}{|D(\zeta)|}\int_{D(\zeta)}|w-u|^{\epsilon}\, dm(u)\le
3^{\epsilon}\]
and therefore
$0\le \psi(w)-\varphi(w)\le 3^{\epsilon}$.

For $|w-\zeta|\ge 2\rho(\zeta)$ (we will write
$v(z)=|w-z|^{\epsilon}$)
we have
\begin{align*}
\psi(w)-\varphi(w) & = \frac{1}{\rho^{\epsilon}(\zeta)}\left\{
\frac{1}{|D(\zeta)|}\int_{D(\zeta)}v(u)\, dm(u)-v(\zeta)\right\}
\\
&
=
\frac{1}{2\pi \rho^{\epsilon}(\zeta)}\int_{D(\zeta)}\left\{ \log \left(   
\frac{\rho(\zeta)}{|u-\zeta|}\right)+\frac{1}{2}\left(\left(\frac{|u-\zeta|}{
\rho(\zeta)}\right)^{2}-1\right)\right\}
\Delta v(u)\, dm(u)
\\
&
\le
\frac{1}{2\pi \rho^{\epsilon}(\zeta)}\int_{D(\zeta)} \log \left(
\frac{\rho(\zeta)}{|u-\zeta|}\right)
\Delta v(u)\, dm(u),
\end{align*}
for the second equality see \cite[section 3.3.]{BO97}.
By \cite[Lemma 5]{MMO03} the last integral is smaller than
\[
\frac{1}{2\pi \rho^{\epsilon}(\zeta)}\int_{D(\zeta)}\Delta v(u)\, dm(u),
\]
times a constant depending only on the doubling constant for
$\Delta v$ (which in turn depends only on
$\epsilon$).
For any $u\in D(\zeta)$ one deduces from $|w-\zeta|\ge 2\rho(\zeta)$ that
$|u-w|\ge \rho(\zeta)$,and
\begin{align*}
\int_{D(\zeta)}
\Delta v(u)\, dm(u)\le \left(
\frac{\epsilon}{2}\right)^{2}\frac{1}{\rho^{2-\epsilon}}(\zeta)m(D(\zeta)),
\end{align*}
so finally
\[
\psi(w)-\varphi(w)\lesssim \frac{\epsilon^{2}}{8}.
\]
\end{proof}

\section{Proof of Theorems~\ref{theo_estimate_kernels} and
\ref{theorem_compactness}}

\begin{proof}[Theorem~\ref{theo_estimate_kernels}]
Let $\{z_i \}$ be a sequence of points in $\C$ and $r>0$ such that $\{ D^r(z_j)
\}$ is a covering of $\C$. Let $\{ \chi_i\}$ be a partition of unity subordinate
to the covering. Let $k_z(\zeta)=\overline{K(z,\zeta)}/\sqrt{K(z,z)}$ be the
normalized reproducing kernel in $\mathcal{F}_\phi^2$. Consider the operator
\[
L^2(e^{-2\phi}) \ni f\mapsto
u_i(z)=k_{z_i}(z)\int_{\C}\frac{f(\zeta)\chi_i(\zeta)}
{(\zeta-z)k_{z_i}(\zeta)}\, dm(\zeta).
\]
By Cauchy-Pompeiu formula one has that $\dbar u_i=f \chi_i$. Then the kernel
\[
G(z,\zeta)=\left(\sum_i \frac{k_{z_i}(z)
\chi_i(\zeta)}{(\zeta-z)k_{z_i}(\zeta)}\right) e^{\phi(\zeta)-\phi(z)}
\]
is such that
\[
u(z)=\int_{\C}e^{\phi(z)-\phi(\zeta)}G(z,\zeta)f(\zeta)\, dm(\zeta)
\]
solves $\dbar u=f$.
Let $z\in \C$ be fixed and $|z-\zeta|\le R \rho(z)$ for some fixed $R\gg 0$,
then there is a finite number of balls of the covering intersecting $D(z)$
and by
Proposition~\ref{prop_coarse_estimate} one has
\[
|G(z,\zeta)|\lesssim |z-\zeta|^{-1}.
\]

Also when $|z-\zeta|\ge R \rho(z)$ there is a finite number of balls in the
covering containing $\zeta$ and this will give us a finite number
of summands in $G$. For one of these terms one has by
Theorem~\ref{proposition_main_estimate} that
\[
|k_{z_i}(z)|=\frac{|K_{z_i}(z)|}{\| K_{z_i} \|}\lesssim
\frac{e^{\phi(z)}}{\rho(z)\exp d_\phi(z,z_i)^\epsilon}
\]
and
\[
\left| \frac{k_{z_i}(z) \chi_i(\zeta)}{(\zeta-z)k_{z_i}(\zeta)}\right|
e^{\phi(\zeta)-\phi(z)}
\frac{e^{\phi(\zeta)-\phi(z)}}{\rho(z)\rho^{-1}(\zeta)e^{\phi(\zeta)}}
\lesssim \frac{1}{\rho(z)\exp d_\phi(z,z_i)^\epsilon},
\]
but as $d_\phi(z,z_i)\sim d_\phi(z,\zeta)$ this gives us the estimate of
$G$.

Now we want to show that the same estimate holds for the kernel $C$.
If $N$ is the canonical solution operator and $M$ is the solution operator
given by the kernel $G$ above, 
one can see that $N=M-PM$ where $P$ stands for the Bergman projection. Then
for $f\in \mathcal{F}_\phi^2$
\[
Nf(z)=\int_\C C(z,\zeta) e^{\phi(z)-\phi(\zeta)}f(\zeta)\, dm(\zeta)
\]
where
\[
C(z,\zeta)=G(z,\zeta)-e^{-\phi(z)}\int_{\C} K(z,\xi)G(\xi,\zeta)e^{-\phi(\xi)}
\, dm(\xi).
\]
Suppose first that $|z-\zeta|\le \rho(z)$.
We split the last integral and use the estimates on $G$ and the Bergman
kernel
\begin{align*}
\int_{\C} & |K(z,\xi)G(\xi,\zeta)|e^{-(\phi(\xi)+\phi(z))}\, dm(\xi)\lesssim
\frac{1}{\rho(z)}\int_{D^K(\zeta)}\rho^{-1}(\zeta)|\xi-\zeta|^{-1}\, 
dm(\xi)
\\
&
+
\frac{1}{\rho^2(\zeta)}\int_{D^K(\zeta)^c}\frac{\rho^{-1}(\xi)}{\exp
d_\phi(\xi,\zeta)^\epsilon}\, dm(\xi)
\end{align*}
and we get that the first integral is bounded by a constant, where $K>1$ is
such that $D(z)\subset D^K(\zeta)$.
Now by Proposition~\ref{bounded_quotient} there exists
$\epsilon'>0$ such that
\begin{align*}
\int_{D^K(\zeta)^c} & \frac{\rho^{-1}(\xi)}{\exp d_\phi(\xi,\zeta)^\epsilon}
\, dm(\xi)\lesssim
\int_{D^K(\zeta)^c\cap
\{\xi:|\zeta-\xi|<\rho(\xi)\}}\frac{\rho^{-1}(\zeta)}{\exp
d_\phi(\xi,\zeta)^\epsilon}\, dm(\xi)
\\
&
+
\int_{D^K(\zeta)^c \cap  \{\xi:|\zeta-\xi|\ge
\rho(\xi)\}}\frac{\rho(\zeta)}{\exp d_\phi(\xi,\zeta)^{\epsilon'}}
\frac{dm(\xi)}{\rho^2(\xi)}\lesssim
\rho(\zeta),
\end{align*}
where for the first integral we use that
$\{\xi:|\zeta-\xi|<\rho(\xi)\}\subset D^{K'}(\zeta)$ for some $K'>0$
and for the second one we use Lemma~\ref{lemma_integrability} together with
Proposition~\ref{proposition_regular_phi}
getting
\[
|C(z,\zeta)|\lesssim |z-\zeta|^{-1},\ \text{when }|z-\zeta|\le
\rho(z).
\]
For $|z-\zeta|> \rho(z)$ and given $0<\eta<1$ we split the integral in the
regions
defined by
\begin{enumerate}
\item[(i)]  $d_\phi(\xi,\zeta)\le \eta d_\phi(z,\zeta)$,
\item[(ii)]  $d_\phi(\xi,z)\le \eta  d_\phi(z,\zeta)$,
\item[(iii)] $d_\phi(\xi,\zeta)> \eta d_\phi(z,\zeta)$ and $d_\phi(\xi,z)>
\eta d_\phi(z,\zeta)$.
\end{enumerate}
In case $(i)$ we have
$d_\phi(z,\zeta)\le d_\phi(z,\xi)+d_\phi(\xi,\zeta)\le d_\phi(z,\xi)+\eta
d_\phi(z,\zeta)$
and
$d_\phi(z,\xi)\le d_\phi(z,\zeta)+d_\phi(\xi,\zeta)\le
(1+\eta)d_\phi(z,\zeta)$ then
\[
(1-\eta)d_\phi(z,\zeta)\le d_\phi(z,\xi)\le (1+\eta)d_\phi(z,\zeta)
\]
and (recall that $|G(\xi,\zeta)|\lesssim \rho^{-1}(\zeta)\exp
(-d_\phi(\xi,\zeta)^\epsilon)$ for $|\xi-\zeta|\ge \rho(\zeta)$)
\begin{align*}
\int_{d_\phi(\xi,\zeta)\le \eta d_\phi(z,\zeta)} &
|K(z,\xi)G(\xi,\zeta)|e^{-(\phi(\xi)+\phi(z))}\, dm(\xi)\lesssim
\frac{1}{\rho(z)}\int_{D^{\tau}(\zeta)}\frac{1}{\rho(\xi)|\xi-\zeta|\exp
d_\phi(z,\xi)^\epsilon}\, dm(\xi)
\\
&
+
\int_{\{d_\phi(\xi,\zeta)\le \eta d_\phi(z,\zeta)\}\cap
D^{\tau}(\zeta)^c}\frac{1}{\rho(\zeta)\rho(z)\rho(\xi)\exp(
d_\phi(z,\xi)^\epsilon+d_\phi(z,\xi)^\epsilon)}\, dm(\xi)
\\
&
\lesssim
\rho^{-1}(z)\exp (-d_\phi(z,\zeta)^{\epsilon})   
\left(1+\frac{1}{\rho(\zeta)}\int_{D^{\tau}(\zeta)^c}\frac{\rho^{-1}(\xi)}{\exp
d_\phi(z,\xi)^\epsilon}\, dm(\xi)\right),
\end{align*}
and the last integral can be bounded as above. An entirely analogous
argument proves case $(ii)$.
Let $A$ be denote the region defined by $(iii)$
(in the estimates which follow the value of the exponent $\epsilon$ may
change from line to line
although it is always strictly positive)
\begin{align*}
\int_{A} & |K(z,\xi)G(\xi,\zeta)|e^{-(\phi(\xi)+\phi(z))}\, dm(\xi)\lesssim
\frac{1}{\rho(\zeta)\rho(z)}\int_{A}\frac{\rho^{-1}(\xi)}{\exp (
d_\phi(z,\xi)^\epsilon+d_\phi(\xi,\zeta)^\epsilon)}\, dm(\xi)
\\
&
\lesssim
\frac{1}{\rho(\zeta)\rho(z)}\left(
\int_{A\cap \{d_\phi(\xi,z)\le
d_\phi(\xi,\zeta)\}}\frac{\rho^{-1}(\xi)}{\exp 2d_\phi(z,\xi)^\epsilon}\,
dm(\xi)
+
\int_{A\cap \{d_\phi(\xi,z)\ge
d_\phi(\xi,\zeta)\}}\frac{\rho^{-1}(\xi)}{\exp 2 d_\phi(\xi,\zeta)^\epsilon}
\, dm(\xi)\right)
\\
&
\lesssim
\frac{1}{\rho(z)\rho(\zeta)}
\int_{A}\frac{\rho(\xi)}{\exp  d_\phi(z,\xi)^\epsilon} d\mu(\xi)
+
\frac{1}{\rho(z)\rho(\zeta)}
\int_{A}\frac{\rho(\xi)}{\exp d_\phi(\xi,\zeta)^\epsilon} d\mu(\xi),
\end{align*}
now we have
\begin{align*}
\frac{1}{\rho(z)}\int_{A} & \frac{\rho(\xi)}{\exp  d_\phi(z,\xi)^\epsilon}
d\mu(\xi)\lesssim
\int_{A} \frac{1}{\exp  d_\phi(z,\xi)^\epsilon} d\mu(\xi)
\lesssim
\int_{d_\phi(z,\xi)>\eta d_\phi(z,\zeta)}
\int_{d_\phi(z,\xi)^\epsilon}^{+\infty}
e^{-t}dt d\mu(\xi)
\\
&
\lesssim
\int_{\eta^\epsilon d_\phi(z,\zeta)^\epsilon}^{+\infty} \mu(\{ \xi:
d_\phi(z,\xi)<t^{1/\epsilon} \}) e^{-t}dt
\lesssim
\int_{\eta^\epsilon d_\phi(z,\zeta)^\epsilon}^{+\infty}t^{\gamma}
e^{-t}dt\lesssim
\frac{1}{\exp d_\phi(z,\zeta)^\epsilon}
\end{align*}
\end{proof}

\begin{proof}[Theorem~\ref{theorem_compactness}]
Let $\{z_j\}$ be a sequence of complex numbers such that $z_j \to
\infty$ for $j\to \infty$. We want to show that $\rho(z_j)\to 0$ when $N$ is
compact. By Theorem~\ref{proposition_main_estimate}
\[
|k_z(\zeta)|=\frac{|K(z,\zeta)|}{\| K_z \|}\lesssim
\frac{e^{\phi(\zeta)}}{\rho(\zeta)\exp d_\phi(z,\zeta )^{\epsilon}}.
\]
Defining holomorphic $(0,1)-$forms $f_j$ and functions $u_j$ as
\[
f_j (z)=k_{z_j}(z)d\bar{z},\qquad u_j(z)=\overline{(z-z_j)}k_{z_j}(z),
\]
then $\dbar u_j=f_j$. Observe that $u_j \in \mathcal{F}^2_\phi$ because of
the above estimate, Lemma~\ref{lemma_integrability}
and Proposition~\ref{proposition_regular_phi}
\begin{align*}
\int_{\C} & |z-z_j|^2|k_{z_j}(z)|^2 e^{-2\phi(z)}\, dm(z)\lesssim
\rho^2(z_j)<\infty.
\end{align*}
Finally, as the reproducing kernels
$\{k_w \}_{w \in \C}$ are dense in $\mathcal{F}_\phi^2$ and
\[
\la u_j , k_{w} \ra=\la (z-z_j)k_{z_j}(z) , k_w(z) \ra=0,
\]
the solution $u_j$ is the canonical solution to $\dbar$ i.e. $u_j=N f_j$.
By hypothesis, the operator $N$ is compact and $\| f_j \|=1$ and therefore
there exist a convergent subsequence of $\{ u_j \}$
(which we denoted as before).

The functions $u_j$ are basically concentrated on $D(z_j)$.
Indeed, by Proposition~\ref{prop_coarse_estimate}
one has $|k_{z_j}(z)|\lesssim  \rho^{-1}(z_j)e^{\phi(z)}$ so
\begin{align*}
\int_{D^r(z_j)} & |(z-z_j)k_{z_j}(z)|^2 e^{-2\phi(z)}\, dm(z)\lesssim
\frac{1}{\rho^2(z_j)}
\int_{D^r (z_j)}|z-z_j|^2\, dm(z)\lesssim  \rho^2(z_j)
\end{align*}
and conversely
by Lemma~\ref{lemma_sub_estimate} (c)
\begin{align*}
\int_{D^r(z_j)} & |(z-z_j)k_{z_j}(z)|^2 e^{-2\phi(z)}\, dm(z)\gtrsim
\int_{D^r(z_j)\setminus D^{r/2}(z_j)} |(z-z_j)k_{z_j}(z)|^2 e^{-2\phi(z)}
\, dm(z)
\\
&
\gtrsim
\rho^4(z_j)\int_{D^r(z_j)\setminus D^{r/2}(z_j)} |k_{z_j}(z)|^2
e^{-2\phi(z)} \frac{dm(z)}{\rho^2(z)} \gtrsim
\rho^4(z_j) |k_{z_j}(z_j)|^2 e^{-2\phi(z_j)}\sim \rho^2(z_j).
\end{align*}
In particular, just because the operator $N$ is bounded, the sequence $\{
\rho(z_j) \}$ has to be bounded.
Also by
Lemma~\ref{lemma_integrability} and
Proposition~\ref{proposition_regular_phi} one has
\begin{align*}
\int_{D^r(z_j)^c}|(z-z_j)k_{z_j}(z)|^2 e^{-2\phi(z)}\, dm(z)\lesssim C_r
\rho^{2}(z_j)
\end{align*}
where $C_r\to 0$ when $r\to \infty$.

The sequence $\{ u_j \}$ is a Cauchy sequence so
\[
\| u_j-u_k \|^2=\| u_j\|^2+\| u_k \|^2+2\Re \la u_j, u_k \ra\to 0,
\]
for $j,k\to \infty$. To complete this part of the proof we have to see that
the scalar product is small also when $z_j$ and $z_k$ are far enough from
each other.
Indeed, given $\epsilon>0$ there exists $r_\epsilon$ such that for $r\ge
r_\epsilon$
\[
\int_{D^r(z_k)^c}|(z-z_k)k_{z_k}(z)|^2 e^{-2\phi(z)}
dz,\int_{D^r(z_j)^c}|(z-z_j)k_{z_j}(z)|^2 e^{-2\phi(z)} dz<\epsilon.
\]
Now let $|z_j-z_k|\gg r_\epsilon \max\{ \rho(z_j),\rho(z_k)  \}$.
The $L^2-$norm of $u_j$ on $D^{r}(z_j)$ is pointwise equivalent to
$\rho(z_j)$ (and this value is bounded above) so applying
H\"older's inequality to
\begin{align*}
|\la u_j, u_{k} \ra | & \le \int_{\C} |z-z_j| |z-z_k||k_{z_j} (z)||k_{z_k}
(z)|e^{-2\phi(z)}\, dm(z)
\\
&
\lesssim \left[ \int_{D^{r_\epsilon} (z_j)} +\int_{D^{r_\epsilon} (z_k)}
+\int_{\C\setminus D^{r_\epsilon} (z_j) \cup D^{r_\epsilon} (z_k)}  \right]
|z-z_j| |z-z_k||k_{z_j} (z)||k_{z_k} (z)|e^{-2\phi(z)}\, dm(z),
\end{align*}
we deduce that the scalar product is arbitrarily small and
\[
\rho^2(z_j)\sim \| u_j \|^2\to 0,\quad j\to \infty.
\]

Suppose now that $\rho(z)\to 0$ when $|z|\to \infty$ and let
\[
M:L^2(e^{-2\phi}) \to L^2 (e^{-2\phi})
\]
be such a solution operator, i.e. $\dbar Mf=f$. If $M$ is compact then the
canonical solution operator will be compact because it can be written as $N=M-P
M$ where $P$ is the Bergman projection.

So all we have to show is that there exists a solution operator for the $\dbar$
problem which is compact. First of all, the operator $M_\delta:
L^2(e^{-2\phi})\to L^2(e^{-2\phi})$ defined as
\[
M_\delta  f(z)=\int_{\{\zeta\in \C:
d_\phi(z,\zeta)<\delta \}} G(z,\zeta)f(\zeta) e^{\phi(z)-\phi(\zeta)}
\, dm(\zeta)
\]
has norm $O(\delta)$ as $\delta\to 0$. Indeed, let $z \in \C$ be fixed, then
\begin{align*}
|M_\delta  f (z)e^{-\phi(z)}|
\le
\| f e^{-\phi} \|_{L^\infty (\C)}
\int_{|z-\zeta|<C \delta \rho(z)} \frac{1}{|z-\zeta|}\, dm(\zeta)\le C
\delta \rho(z) \| f e^{-\phi} \|_{L^\infty (\C)},
\end{align*}
where the constant $C$ only depends on the doubling constant for $\Delta\phi$.
Also
\begin{align*}
\int_{\C}|M_\delta  f (z)|e^{-\phi(z)}\, dm(z)
\lesssim
\delta \| \rho \|_\infty \| f e^{-\phi}\|_{L^1 (\C)},
\end{align*}
and by Marcinkiewicz interpolation theorem, when $\rho$ is bounded, the norm of
the operator from $L^2(e^{-2\phi})$ to $L^2(e^{-2\phi})$ is $O(\delta)$.

We define now (for big $R>0$) the operator $M_\delta^R$ as
\[
M_\delta^R f  (z)=\chi_{B(0,R)}(z) \int_{\{\zeta\in \C: \delta<d_\phi(z,\zeta)
\}} G(z,\zeta)f(\zeta) e^{\phi(z)-\phi(\zeta)}\, dm(\zeta).
\]
This operator is compact because it is Hilbert-Schmidt
\begin{align*}
\int_{B(0,R)} & \int_{\{\zeta\in \C: \delta<d_\phi(z,\zeta)\}}|G(z,\zeta)|^2
\, dm(\zeta)dm(z)
\\
&
\lesssim
\int_{B(0,R)} \frac{1}{\rho^2(z)}\int_{D^{\delta}(z)^c}\frac{1}{\exp(2
d_\phi(z,\zeta)^\epsilon)}\, dm(\zeta)dm(z)
\le O(R^2).
\end{align*}
Finally, for big $R>0$, we define the operator $M^R$ as 
\[
M^R f(z)=\chi_{B(0,R)^c}(z) \int_{\{\zeta\in \C: \delta<d_\phi(z,\zeta) \}}
G(z,\zeta)f(\zeta) e^{\phi(z)-\phi(\zeta)}\, dm(\zeta)
\]
We can control its norm, because
\begin{align*}
|M^R f (z)e^{-\phi(z)}|
\lesssim
\chi_{B(0,R)^c}(z) \rho(z) \| f e^{-\phi} \|_{L^\infty (\C)}
\end{align*}
and therefore
\[
\|e^{-\phi}  M^R f  \|_{L^\infty(\C)}\lesssim \sup_{|z|\ge R}\rho(z)\| f
e^{-\phi} \|_{L^\infty (\C)}.
\]
For the $L^1$ norm
\begin{align*}
\int_{\C} & |M^R f (z)|e^{-\phi(z)}\, dm(z)
\lesssim
\int_{B(0,R)^c}\frac{1}{\rho(z)}\int_{\C}
\frac{1}{\exp(d_\phi(z,\zeta)^\epsilon)} |f(\zeta)|
e^{-\phi(\zeta)}\, dm(\zeta) dm(z)
\\
&
\lesssim
\left(\sup_{|z|\ge R} \rho(z)\right) \int_{\C}|f(\zeta)|
e^{-\phi(\zeta)} \int_{d_\phi(\zeta,z)>\delta}
\frac{1}{\rho^2(z)\exp(d_\phi(z,\zeta)^\epsilon)}\, dm(z)  dm(\zeta),
\end{align*}
the inner integral is finite again because of
Lemma~\ref{lemma_integrability}
combined with Proposition~\ref{proposition_regular_phi}.
Finally, by the Marcinkiewicz interpolation theorem
\[
\|e^{-\phi} M^R f  \|_{L^2(\C)}\lesssim \sup_{|z|\ge R}\rho(z)\| f
e^{-\phi} \|_{L^2 (\C)},
\]
and the norm of $M^R$ goes to $0$ when $R\to \infty$.
So we have that $M=M_\delta + M_\delta^R+M^R$ is compact because
the norm of $M_\delta +M^R $ can be made arbitrarily small and $M_\delta^R$
is compact.
\end{proof}

\begin{proof}[Proposition~\ref{proposition_bounded_no_compact}]
We will use again Marcinkiewicz interpolation theorem. Because of the decay of
$C(z,\zeta)$ we have for $f e^{-\phi} \rho \in L^p$ that 
\[
u(z)=\int_\C C(z,\zeta)f(\zeta)e^{\phi(z)-\phi(\zeta)}\, dm (\zeta),
\]
is a well defined function. Now the estimates on the kernel $C(z,\zeta)$,
\[
\int_{D(z)}\frac{dm(\zeta)}{|z-\zeta|}\lesssim \rho(z),\quad\text{and}\quad
\int_{\C}\frac{dm(\zeta)}{\rho(\zeta)\exp d_\phi(z,\zeta)^\epsilon}\lesssim
\rho(z)
\]
yield $\| u e^{-\phi}\|_{L^\infty (\C)} \lesssim \| f e^{-\phi} \rho
\|_{L^\infty (\C)}$ and $\| u e^{-\phi}\|_{L^1 (\C)} \lesssim \| f e^{-\phi}
\rho \|_{L^1 (\C)}$. The rest of the proof is similar to the proof of Theorem
~\ref{theorem_compactness}.

Assume now that the operator is compact. Let $\{z_j \}$ be a sequence of complex
numbers such that the disks $D(z_j)$ are pairwise disjoint. If
\[
f_j (z)=\frac{k_{z_j}(z)}{\rho(z_j)}d\bar{z},\qquad
u_j(z)=\overline{(z-z_j)}\frac{k_{z_j}(z)}{\rho(z_j)},
\]
one has $\dbar u_j=f_j$ and
\[
\int_{\C}|f_j(z)|^2 e^{-2\phi(z)}\rho(z)\, dm(z)\lesssim 1
\]
and one can extract a converging subsequence of $\{u_j\}$. But as before, from
\[
\| u_j-u_k \|^2=\| u_j\|^2+\| u_k \|^2+2\Re \la u_j, u_k \ra\to 0
\]
we get a contradiccion because $\| u_j \| \sim 1$ and $|\la u_j, u_k \ra|\to 0$
for a fixed $k$ when $j\to \infty$.
\end{proof}

\end{document}